\documentclass[12pt]{amsart}

\usepackage{amssymb}
\usepackage{latexsym}
\usepackage{verbatim,longtable}
\usepackage{fullpage,color}
\usepackage{euscript}

\input {cyracc.def}
\numberwithin{equation}{section}

\usepackage[colorlinks=true,linkcolor=blue,urlcolor=my_color,citecolor=magenta]{hyperref}

 \font\tencyr=wncyr10 
\font\tencyi=wncyi10 
\font\tencysc=wncysc10 
\def\rus{\tencyr\cyracc}
\def\rusi{\tencyi\cyracc}
\def\rusc{\tencysc\cyracc}

\newtheorem{thm}{Theorem}[section]
\newtheorem{conj}[thm]{Conjecture}
\newtheorem{lm}[thm]{Lemma}

\newtheorem{prop}[thm]{Proposition}

\theoremstyle{remark}
\newtheorem{rmk}[thm]{Remark}
\newtheorem*{rem}{Remark}

\theoremstyle{definition}
\newtheorem{ex}[thm]{Example}
\newtheorem{df}{Definition}

\newcommand {\be}{{\mathfrak b}}
\newcommand {\ce}{{\mathfrak c}}

\newcommand {\g}{{\mathfrak g}}

\newcommand {\te}{{\mathfrak t}}
\newcommand {\ut}{{\mathfrak u}}
\newcommand {\z}{{\mathfrak z}}

\newcommand {\cI}{{\mathcal I}}

\newcommand {\BZ}{{\mathbb Z}}

\newcommand {\BR}{{\mathbb R}}
\newcommand {\BC}{{\mathbb C}}

\newcommand{\lb}{\lambda}
\newcommand{\ap}{\alpha}
\newcommand{\vp}{\varphi}

\renewcommand{\le}{\leqslant}
\renewcommand{\ge}{\geqslant}
\newcommand{\curge}{\succcurlyeq}
\newcommand{\curle}{\preccurlyeq}

\newcommand{\eus}{\EuScript}

\newcommand {\ad}{{\mathrm{ad}}}

\newcommand {\hot}{{\mathsf{ht}}}

\newcommand {\rk}{{\mathsf{rk\,}}}

\newcommand {\AN}{\mathfrak{An}}
\newcommand {\anp}{\AN(\eus P)}

\newcommand {\anod}{\AN(\Delta(1)\!)}
\newcommand {\jdod}{\eus J_-(\Delta(1)\!)}
\newcommand {\jdodp}{\eus J_+(\Delta(1)\!)}
\newcommand {\mdt}{\eus M_{\Delta(1)}(t)}

\newcommand {\HN}{\widehat N}
\newcommand {\HW}{\widehat W}
\newcommand {\HV}{\widehat V}
\newcommand {\HP}{\widehat\Pi}
\newcommand {\HD}{\widehat\Delta}

\newcommand {\tri}{\mathfrak{sl}_2}
\newcommand {\GR}[2]{{\textrm{{\bf #1}}}_{#2}}

\newcommand {\ov}{\overline}

\newcommand {\beq}{\begin{equation}}
\newcommand {\eeq}{\end{equation}}

\definecolor{my_color}{rgb}{0,0.5,0.5}
\definecolor{MIXT}{rgb}{0.7,0.1,0.2}
\definecolor{MIX}{rgb}{0.4,0.3,0.6}

\begin{document}
\hfill {{
\scriptsize December 2, 2014}}
\vskip1ex

\title[Weight posets associated with gradings]
{{\color{MIX}Weight posets associated with gradings of simple Lie algebras, Weyl groups, and arrangements of hyperplanes}}
\author[D.\,Panyushev]{Dmitri I.~Panyushev}
\address{Institute for Information Transmission Problems of the R.A.S., 
\hfil\break\indent Bol'shoi Karetnyi per. 19, Moscow 
127994, Russia}
\email{panyushev@iitp.ru}
\address[]{Independent University of Moscow,
Bol'shoi Vlasevskii per. 11, 119002 Moscow, \ Russia}
\maketitle


\section{Introduction}  
\label{sec:intro}

\noindent
The set of weights of a finite-dimensional representation of a reductive Lie algebra has a natural poset structure ("weight poset"). Studying certain combinatorial problems related to antichains in weight posets,
we realised that the best setting is provided by the representations associated with $\BZ$-gradings of simple
Lie algebras \cite{ja}. This article, which can be regarded as a sequel to \cite{ja}, is devoted to a general 
theory of ideals (antichains) in the corresponding weight posets. Although the subject has interesting representation-theoretic aspects, we work here almost exclusively in the combinatorial setup. Specifically, our main object is going to be a $\BZ$-graded root system.

Let $V$ be an $n$-dimensional Euclidean space, with inner product $(\ ,\ )$, and let $\Delta$ be an 
irreducible, crystallographic root system spanning $V$. We refer to \cite{bour,hump} for basic 
definitions and properties of root systems. Let $\Delta^+$ be a set of positive roots and 
$\Pi=\{\ap_1,\dots,\ap_n\}$ the set of {simple roots} in $\Delta^+$. The usual partial order
``$\curle$'' in $\Delta^+$ is defined by the requirement that $\gamma$ covers $\mu$ if and only if 
$\gamma-\mu\in\Pi$. 
A $\BZ$-{\it grading\/} of $\Delta$ is a disjoint union $\Delta=\bigsqcup_{i\in \BZ}\Delta(i)$ such that
if $\gamma_1\in \Delta(i_1)$, $\gamma_2\in \Delta(i_2)$, and $\gamma_1+\gamma_2$ is a root, then
$\gamma_1+\gamma_2\in \Delta(i_1+i_2)$. Then $\Delta(0)$ is a root system in its own sense. 
We always assume that $\Delta^+$ is {\it compatible\/} with $\BZ$-grading, which means that
\beq    \label{eq:comptible-gr}
    \Delta^+=\Delta(0)^+\sqcup \Delta(1)\sqcup\Delta(2)\sqcup\ldots ,
\eeq
where $\Delta(0)^+$ is a set of positive roots in $\Delta(0)$. Then $\Pi=\sqcup_{i\ge 0}\Pi(i)$, where
$\Pi(i)=\Pi\cap\Delta(i)$, and $\Pi(0)$ is a set of simple roots for $\Delta(0)$.
Each $\Delta(i)$, $i\ge 1$, can be regarded as a sub-poset of $\Delta^+$, and 
we are primarily interested in the poset $\Delta(1)$. 

Let $\jdod$ be the set of lower (=\,order) ideals in $\Delta(1)$. We relate $\jdod$ to certain 
elements in the {Weyl group} $W$ of $\Delta$ and certain hyperplane arrangements inside the 
Coxeter arrangement of $\Delta$. The Weyl group of $\Delta(0)$, $W(0)$, is a parabolic subgroup of 
$W$.  Let $W^0$ be the set of minimal length coset
representatives for $W/W(0)$. It is known that 
\beq     \label{eq:w0-def}
         W^0=\{w\in W \mid w(\ap)\in \Delta^+ \quad \forall \ap\in\Delta(0)^+\} ,
\eeq
see \cite[1.10]{hump}.
Let $N(w)=\{\gamma\in \Delta^+\mid -w(\gamma) \in \Delta^+\}$ be the {\it inversion set\/} of $w\in W$ and $w\mapsto \ell(w)=\# N(w)$ the length function on $W$. 
By a classical result of {Kostant}~\cite[Prop.\,5.10]{ko61}, $M\subset \Delta^+$ is the inversion set of some $w$ if and only if 
both $M$ and $\Delta^+\setminus M$ are closed under addition. Such an $M$ is said to be
{\it bi-convex}.

Our basic results on the lower ideals in $\Delta(1)$ and related elements of $W^0$ are presented in Section~\ref{sect:min-and-max-els}.
It is readily seen that if $w\in W^0$, then
$I_w:=N(w)\cap\Delta(1)$ is a lower ideal of $\Delta(1)$, which yields the map
\[
   \tau: W^0\to \jdod, \quad w\mapsto \tau(w):=I_w .
\]
For any $I\in \jdod$, we construct two extreme bi-convex subsets of $\Delta^+$ that belong to $\bigcup_{k\ge 1}\Delta(k)$ and whose $1$-component is $I$, see Theorems~\ref{thm:biconvex-} and 
\ref{thm:biconvex+}. This implies that $\tau$ is onto and $\tau^{-1}(I)$ contains a unique element of
minimal and  of 
maximal length. These two elements of $W$ are said to be the {\it minimal\/} and the 
{\it maximal elements\/} of $I$, denoted $w_{I,{\rm min}}$ and $w_{I,{\rm max}}$, respectively.
Furthermore, we observe that $\tau^{-1}(I)$ is an interval w.r.t. the weak Bruhat order ``$\leq$'' on $W^0$;
that is, 
$
    \tau^{-1}(I)=\{w\in W^0\mid w_{I,{\rm min}}\leq w\leq w_{I,{\rm max}} \}
$, see Theorem~\ref{thm:interval-weak-Br}. 

Let $W^0_{\rm min}$ (resp. $W^0_{\rm max}$) be the subset 
of $W^0$ that consists of the minimal (resp. maximal) elements of all lower ideals. 
We provide a characterisation of each subset that
does not refer to lower ideals. Set $\Delta({\ge} k)=\sqcup_{j\ge k} \Delta(j)$ and $\Delta({\le} k)=\sqcup_{j\le k} \Delta(j)$. Then 
\begin{gather*}
  W^0_{\rm min}=\{w\in W^0\mid w^{-1}(\ap)\in \Delta({\ge}{-}1) \ \text{ for all }\ap\in\Pi \},  
   \text{  (Theorem~\ref{thm:0-min-apriori})};
  \\
  W^0_{\rm max}=\{w\in W^0\mid w^{-1}(\ap)\in \Delta({\le}1) \ \text{ for all }\ap\in\Pi \} ,
  \text{ (Theorem~\ref{thm:0-max-apriori})}.
\end{gather*}
We also point out a connection between an involution on $\jdod$, involution on $W^0$, and the subset
$W^0_{\rm min}$ and $W^0_{\rm max}$ (Proposition~\ref{lem:min-max-in-W^0}).

As an application of our minimal/maximal elements, we describe the antichains related to the lower ideals. 
Let $\min(M)$ and $\max(M)$ denote the minimal and  maximal elements of a subset $M$ w.r.t. the 
poset structure of $\Delta(1)$. For $I\in\jdod$, one may consider two antichains:  
$\max(I)$ and $\min(\Delta(1)\setminus I)$. 
Given $\gamma\in\Delta(1)$, our result is that

\textbullet\quad  
 $\gamma\in \max(I)$ if and only if $w_{I,\text{min}}(\gamma)\in -\Pi$, see Theorem~\ref{thm:max-roots};

\textbullet\quad  
$\gamma\in \min(\Delta(1)\setminus I)$ if and 
only if $w_{I,\text{max}}(\gamma)\in \Pi$, see  Theorem~\ref{thm:min-roots}.

\vskip1ex\noindent
Associated with $\Delta^+$ and $\Delta(0)^+$, there are two open dominant chambers, 
$\eus C^o=\{v\in V\mid (v,\ap)>0 \ \ \forall \ap\in\Pi\}$ and $\eus C(0)^o=\{v\in V\mid (v,\ap)> 0\ \ \forall\ap\in\Pi(0)\}$. The chambers $w(\eus C^o)$, $w\in W$, are said to be {\it small}.
Let $\eus H_\gamma$ denote the hyperplane in $V$ orthogonal to $\gamma\in\Delta^+$. 
The hyperplanes $\eus H_\gamma$ with $\gamma\in\Delta(1)$ dissect $\eus C(0)^o$ into certain 
regions, and we prove that there is a natural bijection between  $\jdod$ and the set of these
regions. Moreover, if $\eus R_I^o\subset \eus C(0)^o$ is the open 
region corresponding to $I$, then $w_{I,{\rm min}}^{-1}(\eus C^o)$ is the unique small chamber in $\eus R_I^o$ closest
to $\eus C^o$ and $w_{I,{\rm max}}^{-1}(\eus C^o)$ is the unique small chamber in $\eus R_I^o$ farthest from
$\eus C^o$ (Theorem~\ref{thm:min-max-&-chambers}).
This result prompts considering the hyperplane arrangement 
$\eus A_\Delta(0,1)=\{\eus H_\gamma \mid \gamma\in\Delta(0)^+\cup \Delta(1)\}$ in $V$.

It is well known that the whole {\it Coxeter arrangement\/} 
$\eus A_\Delta=\{\eus H_\gamma \mid \gamma\in\Delta^+\}$
is free and its exponents are just the usual exponents of $W$~\cite[Ch.\,6]{OS}. 
We conjecture that the arrangement $\eus A_\Delta(0,1)$ is also free and its exponents are determined by certain partition 
associated with $\Delta(0)^+\cup \Delta(1)$ (Conjecture~\ref{conj:free-arr}). Actually, this is a special
case of a more general conjecture that is discussed in the Introduction of~\cite{som-tym}.
Moreover, by \cite[Theorem\,11.1]{som-tym}, that general conjecture and hence our
Conjecture~\ref{conj:free-arr} are true if $\Delta$ is classical 
or of type $\GR{G}{2}$. For $\gamma\in\Delta$, let 
$[\gamma:\ap_i]$ be the coefficient of $\ap_i$ in the expression of $\gamma$ via the simple roots.
The {\it height\/} of $\gamma$ is $\hot(\gamma)=\sum_{i=1}^n [\gamma:\ap_i]$. We deduce from Conjecture~\ref{conj:free-arr} that 
\[
  \#\jdod=\prod_{\gamma\in\Delta(1)}\frac{\hot(\gamma)+1}{\hot(\gamma)} .
\]
This equality has also been proved in \cite{ja}, by {\sl ad hoc\/} methods, for the abelian and extra-special gradings of $\Delta$ (see Section~\ref{subs:classes} for their definitions).

An inspiring observation is that, to a great extent, the theory of lower ideals in $\Delta(1)$
is parallel (similar) to the theory of upper (=\,{\sl ad}-nilpotent) ideals in the poset $(\Delta^+,\curle)$. The  
latter will be referred to as the {\it affine theory}, because it requires the use of the affine Weyl group $\HW$ 
and the affine root system $\HD$. We discuss this parallelism in Section~\ref{sect:versus}.

In Appendix~\ref{app:A}, we give a case-free proof of an observation in \cite[Prop.\,3.1]{som-tym} to the 
effect that certain sequence associated with an upper ideal of $\Delta^+$ is, actually, 
a partition. This fact is also needed for Conjecture~\ref{conj:free-arr}.

{\small
{\bf Acknowledgements.} Part of this work was done while I was able to use
rich facilities of the Max-Planck Institut f\"ur Mathematik (Bonn).
}

\section{Weight posets and gradings of simple Lie algebras}   
\label{sec:general}

\noindent
Let $(\eus P, \curle)$ be a finite poset. A {\it lower\/} (resp. {\it upper\/}) {\it ideal\/} $I$ is a subset of $\eus P$ 
such that if $\mu\in I$ and $\nu\curle \mu$ (resp. $\nu\curge\mu$), then $\nu\in I$.
Let  $\eus J_-(\eus P)$ be the set of lower ideals, $\eus J_+(\eus P)$ the set of 
upper ideals, and $\anp$ the set of antichains in $\eus P$.
For any $M\subset \eus P$, let $\min(M)$ (resp. $\max(M)$) denote the set of minimal (resp. maximal)
elements of $M$ with respect to `$\curle$'. The following three maps set up bijections between the respective pairs of sets:
\begin{gather*}
I\in \eus J_-(\eus P) \mapsto \max(I)\in \anp , \quad
I\in \eus J_+(\eus P) \mapsto \min(I)\in \anp , \\
I\in \eus J_-(\eus P) \mapsto I^{\boldsymbol c}:=\eus P\setminus I\in \eus J_+(\eus P) .
\end{gather*}

\noindent
Both $\eus J_-(\eus P)$ and $\eus J_+(\eus P)$ are graded posets under inclusion, with the
rank function $I\mapsto \# I$. The rank-generating function of either of them is
\[
\eus M_\eus P(t)=\sum_{I\in \eus J_-(\eus P)} t^{\#I} .
\] 
It is also called the $\eus M$-{\it polynomial\/} of $\eus P$ in \cite{ja}. Clearly, $\eus M_\eus P(1)=
\# \eus J_-(\eus P)=\# \anp$.

\subsection{Gradings of simple Lie algebras and root systems} 
\label{subs:gradings}
Although we are primarily interested in combinatorics of posets related to $\BZ$-gradings of 
root systems, it is instructive and helpful 
to keep in mind that a $\BZ$-grading of $\Delta$ is an offspring of a $\BZ$-grading of the corresponding 
simple Lie algebra $\g$. This provides a broader perspective and adds some geometric flavour and 
intuition to one's considerations.
(We refer to \cite[Ch.\,3,\,\S\,3]{t41} for generalities on gradings of semisimple Lie algebras.)

Let $\g=\ut^-\oplus\te\oplus\ut$ be a fixed triangular decomposition, where $\te$ is a Cartan
subalgebra of $\g$. The associated root system $\Delta(\g,\te)$ is $\Delta$, and $V=\te^*_{\BR}$ is the 
$\BR$-span of $\Delta$ in $\te^*$. If $\g_\gamma$ is the root space for $\gamma\in\Delta$, then
$\ut =\bigoplus_{\gamma\in\Delta^+} \g_\gamma$. Write $s_\gamma$ for the 
reflection in $W$ with respect to $\gamma\in \Delta$. Let $\theta$ be the {\it highest root\/} in $\Delta^+$.
Recall that $\hot(\theta)=h-1$, where $h$ is the {\it Coxeter number\/} of $\Delta$.

Let $\g=\bigoplus_{i\in\BZ}\g(i)$ be a $\BZ$-grading. Since any derivation of $\g$ 
is inner,  we have $\g(i)=\{x\in\g \mid [\tilde h,x]=ix\}$ for a unique (semisimple) element $\tilde h\in \g(0)$. 
The element $\tilde h$ is said to be {\it defining\/} for the grading in question. 
Here $\g(0)$ is the centraliser of $\tilde h$, hence a reductive Lie algebra.
Without loss of generality, one may assume that $\tilde h\in\te$ and $\ap(\tilde h)\ge 0$ for all $\ap\in\Pi$.
Then $\te\subset\g(0)$, $\g(0)=(\g(0)\cap\ut^-)\oplus\te \oplus (\g(0)\cap\ut )$
is a triangular decomposition of $\g(0)$, and 
\[   \ut =(\g(0)\cap\ut )\oplus \g(1)\oplus\g(2)\oplus \dots . 
\]
If $\Delta(i)=\{\gamma\in \Delta\mid \gamma(\tilde h)=i\}$, then $\Delta(i)$ is the set of roots of $\g(i)$,  
and $\Delta=\bigsqcup_{i\in\BZ} \Delta(i)$ is a compatible $\BZ$-grading of $\Delta$ in the sense of Introduction, i.e., Eq.~\eqref{eq:comptible-gr} holds.
We also have $\Pi=\bigsqcup_{i\ge 0} \Pi(i)$, where  $\Pi(i)=\{\ap\in\Pi\mid \ap(\tilde h)=i\}$, and $\Pi(0)$ is the set of simple roots in $\Delta(0)^+=\Delta(0)\cap \Delta^+$.

Each $\g(i)$ is a $\g(0)$-module, and therefore $\Delta(i)$ has a natural poset structure as the set of 
weights of a $\g(0)$-module. In case of compatible gradings, this {\it weight poset\/} structure on 
$\Delta(i)$ coincides with the restriction of `$\curle$' to $\Delta(i)$, see \cite[Remark\,2.9]{ja}. 
More precisely, if $\gamma,\gamma'\in \Delta(i)$, 
then $\gamma$ covers $\gamma'$ if and only if $\gamma-\gamma'\in\Pi(0)$. Therefore, $\gamma'\curle
\gamma$ if and only if $\gamma-\gamma'$ is a nonnegative integer linear combination of $\Pi(0)$. 

Set $\be(0)^-=(\g(0)\cap\ut^-)\oplus\te$ and $\be(0)=\te \oplus (\g(0)\cap\ut)$. These are two 
opposite Borel subalgebras of $\g(0)$. A link between combinatorics and geometry is provided by the 
following simple observation, which we do not pursue in this article. 

\begin{prop}
There is a bijection between the lower (resp. upper) ideals of $(\Delta(1),\curle)$ and the $\be^-(0)$-stable
(resp. $\be(0)$-stable) subspaces of\/ $\g(1)$.
\end{prop}
\begin{proof}
If $I\in\jdod$ or $I\in\jdodp$, then $\ce_I=\bigoplus_{\gamma\in I}\g_\gamma$ is the corresponding
$\be^-(0)$-stable or $\be(0)$-stable subspace of $\g(1)$. The details are left to the reader.
\end{proof}
  
\subsection{Standard $\BZ$-gradings}
\label{subs:g1-dostat}
By \cite[\S\,1.2, \S\,2.1]{vi76}, if one is interested in possible $\g(0)$-modules $\g(i)$,  and hence in possible posets $\Delta(i)$, 
then it suffices to consider the $\g(0)$-modules $\g(1)$ for all semisimple $\g$. 
(For $i >1$,  the problem is reduced to considering the induced $\BZ$-grading of
a certain semisimple subalgebra of $\g$.)
For this reason, it suffices to consider defining elements $\tilde h\in\te$ such that $\ap(\tilde h)\in\{0,1\}$, i.e., 
$\Pi=\Pi(0)\sqcup\Pi(1)$.
The corresponding $\BZ$-gradings (of both $\g$ and $\Delta$) are said to be {\it standard}.
More precisely, if $\#\Pi(1)=k$, then we call it a $k$-{\it standard\/} grading.
A standard $\BZ$-grading can be represented by the Dynkin diagram of $\g$,
where the vertices in $\Pi(1)$ are coloured.
If $\Pi(1)=\{\ap_{i_1},\dots,\ap_{i_k}\}$, then the $\ap_{i_j}$'s are precisely 
the lowest weights of the simple $\g(0)$-modules in $\g(1)$, the centre of $\g(0)$ is $k$-dimensional, 
and $\g(1)$ is a direct sum of $k$ simple $\g(0)$-modules. 
In this case, the poset  $\Delta(1)$ is the disjoint union of $k$ subposets corresponding to the simple summands of $\g(1)$. Therefore, all enumerative problems for $\Delta(1)$ reduce to $1$-standard
gradings.

The weight posets $\Delta(i)$, $i>0$, can be visualised as follows. Let 
$\eus H(\Delta^+)$ be the {\it Hasse diagram\/} of $(\Delta^+, \curle)$. If $\gamma'-\gamma=\ap\in\Pi$, 
then the edge connecting $\gamma$ and $\gamma'$ in $\eus H(\Delta^+)$ is said to be {\it of type\/} $\ap$. 
Given a standard $\BZ$-grading of $\g$, let us remove from $\eus H(\Delta^+)$ all the edges of types 
from $\Pi(1)$. This yields a disconnected graph. Each connected component of it is the Hasse 
diagram of either the set of positive roots of a simple factor of $\g(0)$ (if it contains roots from $\Pi(0)$) 
or the weight poset of a simple $\g(0)$-module in some $\g(i)$, $i>0$. The set of weights of a simple $\g(0)$-module in some $\g(i)$, $i\ge 1$, is precisely the set of roots $\gamma$ with fixed values
$[\gamma : \ap]$ for all $\ap\in\Pi(1)$, see e.g.~\cite[3.5]{t41}. 

\subsection{Special classes of $\BZ$-gradings}  
\label{subs:classes}

In \cite[Sect.\,3, 4]{ja}, we considered in details the following two classes of $\BZ$-gradings of $\g$ and hence of $\Delta$:

{\it\bfseries The abelian case}: $\g=\g(-1)\bigoplus\g(0)\bigoplus\g(1)$. 
\\ Here $\g(0)\bigoplus\g(1)$ 
is a parabolic subalgebra and $\g(1)$ is its abelian nilradical.  In this case  $\g(1)$ is a simple 
$\g(0)$-module and therefore such a grading is $1$-standard. If 
$\Pi(1)=\{\tilde\ap\}$, then upon the identification of $\te_\BR$ and 
$\te^*_\BR$, the defining element $\tilde h$ appears to be the minuscule fundamental 
weight $\vp_{\tilde\ap}^\vee$ of the dual root system $\Delta^\vee$. 
As is well known, the admissible simple roots $\tilde\ap$ are characterised by 
the property that $[\theta:\tilde\ap]=1$~\cite[Ch.\,VIII,\ \S\,7, n$^0$3]{bour7-8}.

\vskip1.2ex \indent
{\it\bfseries The extra-special case}: 
$\g=\g(-2)\bigoplus\g(-1)\bigoplus\g(0)\bigoplus\g(1)\bigoplus\g(2)$ and $\dim\g(2)=1$. 
\\    
Any simple Lie algebra has a unique, up to conjugation, $\BZ$-grading of this form, and without loss 
of generality, we may assume that $\Delta(2)=\{\theta\}$. Upon the 
identification of $\te_\BR$ and $\te^*_\BR$, the defining
element $\tilde h$ is recognised as the coroot $\theta^\vee$.
That is, $\Delta(i)=\{\gamma\in\Delta\mid (\gamma,\theta^\vee)=i\}$ and $W(0)$ is the stabiliser of 
$\theta$ (or $\theta^\vee$) in $W$.
Since  here
$\Pi(1)=\{\ap\in\Pi\mid (\gamma,\theta^\vee)\ne 0\}$, we see that 
$\g(1)$ is a simple $\g(0)$-module if and only if $\theta$ is a multiple of a fundamental weight, i.e.,
$\g$ is not of type $\GR{A}{n}$.

The following simple lemma is one of our main tools for inductive arguments in subsequent sections.
\begin{lm}    \label{lem:tri-kornya}
Suppose that the roots $\mu,\nu_1,\nu_2$ have the property that $\nu_1+\nu_2\in\Delta$ and 
$\mu+\nu_1+\nu_2\in \Delta$. Then $\mu+\nu_1$ or $\mu+\nu_2$ is also a root.
\end{lm}
\begin{proof}
1) \ If $(\mu+\nu_1+\nu_2,\nu_1+\nu_2)>0$, then $(\mu+\nu_1+\nu_2,\nu_1)>0$ or
$(\mu+\nu_1+\nu_2,\nu_2)>0$. Hence $\mu+\nu_2$ or $\mu+\nu_1$ is a root.

2) \ If $(\mu+\nu_1+\nu_2,\nu_1+\nu_2)\le 0$, then $(\mu,\nu_1+\nu_2)<0$. 
Hence $(\mu,\nu_1)<0$ or $(\mu,\nu_2)<0$, i.e., again $\mu+\nu_1$ or $\mu+\nu_2$ is a root.
\end{proof}

\section{Elements of $W^0$ associated with the lower ideals in $\Delta(1)$}  
\label{sect:min-and-max-els}

\noindent
In this section, $\Delta=\bigsqcup_{i\in \BZ}\Delta(i)$ is a $\BZ$-grading and
$\Delta^+=\Delta(0)^+\cup\Delta({\ge}1)$. Recall that 
$I\in \jdod$ if and only if whenever $\gamma\in I$, $\mu\in\Delta(0)^+$, and 
$\gamma-\mu\in \Delta$, then $\gamma-\mu\in I$. Then $I^{\boldsymbol c}:=\Delta(1)\setminus I\in\jdodp$. That is, 
if $\gamma\in I^{\boldsymbol c}$, $\mu\in\Delta(0)^+$, and 
$\gamma+\mu\in \Delta$, then $\gamma+\mu\in I^{\boldsymbol c}$.

For any $I\subset \Delta^+$, we set $I^1=I$ and if $I^{k-1}\ne \varnothing$, then
$I^k=(I+I^{k-1})\cap \Delta$ for $k\ge 2$. Then
$\langle I\rangle:=\bigcup_{k\ge 1} I^k \subset \Delta^+$.

\begin{lm}   \label{lem:<I>closed}
 $\langle I\rangle$ is a closed subset of $\Delta^+$.
\end{lm}
\begin{proof}
Suppose that $\gamma_i\in I^{k_i}$, $i=1,2$ and $\gamma_1+\gamma_2$ is a root. Our goal is to prove
that $\gamma_1+\gamma_2\in I^{k_1+k_2}$. Without loss 
of generality, we may assume that $k_1\le k_2$. Arguing by 
induction, we assume that the required property holds for all $(k'_1,k'_2)$ such that either $k'_1+k'_2<
k_1+k_2$ or $k'_1+k'_2=k_1+k_2$ and $k'_1<k_1$.

\textbullet \quad If $k_1=1$, then $\gamma_1+\gamma_2\in I^{k_1+k_2}$ by the very definition of 
$I^k$.

\textbullet \quad If $k_1>1$, then $\gamma_1=\gamma'_1+\gamma''_1$ with $\gamma'_1\in I$ and
$\gamma''_1\in I^{k_1-1}$. By Lemma~\ref{lem:tri-kornya}, we then have $\gamma'_1+\gamma_2\in\Delta$ or $\gamma''_1+\gamma_2\in\Delta$. Hence, by the induction assumption, 
$\gamma'_1+\gamma_2\in I^{k_2+1}$ or $\gamma''_1+\gamma_2\in I^{k_1+k_2-1}$; and in either case
we also conclude that $\gamma'_1+\gamma''_1+\gamma_2\in I^{k_1+k_2}$.
\end{proof}

Now, let us turn to the case in which $I\subset\Delta(1)$.
Then $I^k\subset \Delta(k)$ for all $k\ge 1$. Consequently, $\langle I\rangle\subset \Delta({\ge}1)$.

\begin{prop}    \label{prop:0-ideal}
If $I\in \jdod$, then $I^k\in \eus J_-(\Delta(k)\!)$ 
for any $k\ge 1$. Likewise, if $I\in \jdodp$, then $I^k\in \eus J_+(\Delta(k)\!)$ for any $k\ge 1$.
\end{prop}
\begin{proof}
Argue by induction on $k$ and use Lemma~\ref{lem:tri-kornya}.
\end{proof}

\begin{thm}  \label{thm:biconvex-}
If $I\in \jdod$, then $\langle I\rangle$ is a bi-convex subset of $\Delta^+$.
\end{thm}   
\begin{proof}
By Lemma~\ref{lem:<I>closed},  $\langle I\rangle$ is closed. Set 
$(I^k)^{\boldsymbol c}=\Delta(k)\setminus I^k$ for $k\ge 1$. 
Then 
\[
  \ov{ \langle I\rangle}:=\Delta^+\setminus \langle I\rangle=\Delta(0)^+\cup I^{\boldsymbol c}\cup (I^2)^{\boldsymbol c} \cup \dots ,
\]
and our goal is to prove that $\ov{ \langle I\rangle}$ is closed, too. 
Assuming that this is not the case, one can find 
$\mu',\mu''\in \ov{ \langle I\rangle}$ such that $\mu'+\mu''\in  \langle I\rangle$.
Since $\Delta(0)^+$ is closed and each $(I^k)^{\boldsymbol c}$ is an upper ideal 
(use Proposition~\ref{prop:0-ideal}!), one has to only 
consider the case in which neither $\mu'$ nor $\mu''$ belong to $\Delta(0)^+$.
Specifically, assume that $\mu'\in (I^i)^{\boldsymbol c}$ and $\mu''\in (I^j)^{\boldsymbol c}$ with $i,j\ge 1$, but $\mu'+\mu''\in I^{i+j}$. 
Arguing by induction, we may assume that $i+j$ is the smallest integer with such property.
By the recursive definition of $\langle I\rangle$, one has
\[
    \mu'+\mu''=\gamma_1+\gamma_{i+j-1}\in I^{i+j} , 
\]
where $\gamma_k\in I^k$. Since 
$(\mu'+\mu'',\gamma_1+\gamma_{i+j-1})>0$, we may assume that, say,  $(\mu',\gamma_1)>0$ and 
hence $\nu:=\mu'-\gamma_1=\gamma_{i+j-1}-\mu'' \in \Delta(i-1)$. Now, there are two possibilities for $i$.

{\it\bfseries (a)} \ $i=1$. Then $\nu\in\Delta(0)$. If $\nu\in\Delta(0)^+$, then
$\mu'=\gamma_1+\nu\in I$. If $\nu\in\Delta(0)^-$, then $\mu''=\gamma_j-\nu \in I^j$. In either case, this contradicts the assumption on $\mu',\mu''$.

{\it\bfseries (b)} \ $i>1$. Then $\nu=\mu'-\gamma_1\in \Delta(i-1)\subset \Delta^+$. If $\nu\in I^{i-1}$, then
$\mu'\in I^i$, a contradiction. If $\nu\in (I^{i-1})^{\boldsymbol c}$, then 
\[
    (\mu'-\gamma_1)+\mu''=\gamma_{i+j-1} \in I^{i+j-1},
\]
which contradicts the minimality of $i+j$.

Thus, $\ov{ \langle I\rangle}$ is closed, and we are done. 
\end{proof}

\begin{thm}  \label{thm:biconvex+}
If $I\in \jdod$, then $\Delta({\ge}1)\setminus \langle I^{\boldsymbol c}\rangle$
is a bi-convex subset of $\Delta^+$.
\end{thm}   
\begin{proof}
All the necessary ideas are already contained in the previous proof. 
\par
1. The complement in $\Delta^+$
of the indicated subset is $\langle I^{\boldsymbol c}\rangle\cup \Delta(0)^+$. Here
$\langle I^{\boldsymbol c}\rangle=\bigcup_{k\ge 1} (I^{\boldsymbol c})^k $ is closed by Lemma~\ref{lem:<I>closed}, and, 
by Proposition~\ref{prop:0-ideal}, each 
$(I^{\boldsymbol c})^k$ is an upper ideal of $\Delta(k)$. Therefore,
$\langle I^{\boldsymbol c}\rangle\cup \Delta(0)^+$ is closed, too.
\par
2. To prove that $\Delta({\ge}1)\setminus \langle I^{\boldsymbol c}\rangle=
\bigcup_{k\ge 1}\bigl(\Delta(k)\setminus (I^{\boldsymbol c})^k\bigr)$ is closed, one uses the fact
that each $\Delta(k)\setminus (I^{\boldsymbol c})^k$ is a lower ideal and
repeats {\sl mutatis mutandis} the inductive argument of the previous proof.
\end{proof}

By Theorem~\ref{thm:biconvex-}, there is a unique $w\in W$ such that $N(w)=\langle I\rangle$. In particular,  

 ${\color{MIX}(\clubsuit)}$ \hfil $N(w)\cap \Delta(1)=I$. \hfil

\noindent Since 
$N(w)\subset\Delta({\ge}1)$, we also have $w\in W^0$. Furthermore, if $w'\in W^0$ also satisfies 
${\color{MIX}(\clubsuit)}$, then $N(w')\supset \langle I\rangle=N(w)$. Thus, 
$w$ is the unique element of {\bf minimal} length in $W^0$ such that the $1$-component of $N(w)$ is $I$.
We shall say that $w$ is the {\it minimal element of} $I$ and denote it by $w_{I,\text{min}}$.
Likewise, by Theorem~\ref{thm:biconvex+}, there is a unique $\tilde w\in W^0$ such that 
$N(\tilde w)=\Delta({\ge}1)\setminus \langle I^{\boldsymbol c}\rangle$. Clearly, 
the $1$-component of $N(\tilde w)$ is $I$.
Furthermore, if $w'\in W^0$ also satisfies 
${\color{MIX}(\clubsuit)}$, then $\Delta^+\setminus N(w')\supset 
\langle I^{\boldsymbol c}\rangle\cup \Delta(0)^+=\Delta^+\setminus N(\tilde w)$.
Thus, $\tilde w$ is the unique element of {\bf maximal} length in $W^0$ such that the $1$-component of 
$N(\tilde w)$ is $I$.
For this reason, we say that $\tilde w$ is the {\it maximal element of} $I$ and denote it by $w_{I,\text{max}}$.

\begin{rmk}
It is readily seen that if $w\in W^0$, then $I_w:=N(w)\cap \Delta(1)$ is a lower ideal in $\Delta(1)$.
This provides the natural map $\tau: W^0 \to \jdod$, $w\mapsto I_w$.
An offspring of Theorems~\ref{thm:biconvex-} and \ref{thm:biconvex+} is that  $\tau$ is onto and we have
two sections
$\boldsymbol{s}_{min}, \boldsymbol{s}_{max}: \jdod\to W^0$ for $\tau$, where
$\boldsymbol{s}_{min}(I)= w_{I,{\rm min}}$ \ and \ $\boldsymbol{s}_{max}(I)=w_{I,{\rm max}}$. 
\end{rmk}

Recall that the {\it weak Bruhat order\/} ``$\leq$'' on (any subset of) $W$ is defined by the condition that
$w\leq w'$ if and only if $N(w)\subset N(w')$. As a consequence of preceding results, we obtain the following interesting fact.

\begin{thm}   \label{thm:interval-weak-Br}
For any $I\in \jdod$, $\tau^{-1}(I)$ is an interval with respect to the weak Bruhat order in
$W^0$. Namely, $\tau^{-1}(I)=\{w\in W^0\mid w_{I,{\rm min}} \leq w\leq w_{I,{\rm max}}\}$.
\end{thm}
\begin{proof}
If $w\in \tau^{-1}(I)$, then $N(w)\cap \Delta(1)=I$ and hence 
\[
   N(w_{I,{\rm min}})=\langle I\rangle \subset N(w)\subset 
   \Delta({\ge}1)\setminus \langle I^{\boldsymbol c}\rangle=N(w_{I,{\rm max}}) ,
\]
in view of the definitions of $w_{I,{\rm min}}$ and $w_{I,{\rm max}}$. That is, $w_{I,{\rm min}} \leq w\leq w_{I,{\rm max}}$.

The other implication is obvious.
\end{proof}

\begin{df}   \label{def:min-&-max}
The set of {\it minimal elements of\/} $W^0$ is
$W^0_{\rm min}=\{w_{I,{\rm max}}\mid I\in \jdod \}$;
\\ \indent
The set of {\it maximal elements of\/} $W^0$ is
$W^0_{\rm max}=\{w_{I,{\rm max}}\mid I\in \jdod \}$. 
\end{df}

Our next aim is to provide alternative descriptions of the sets $W^0_{\rm min}=\boldsymbol{s}_{min}(\jdod)$ 
and  $W^0_{\rm max}=\boldsymbol{s}_{max}(\jdod)$.

\begin{thm}    \label{thm:0-min-apriori}
 $W^0_{\rm min}=\{w\in W^0\mid w^{-1}(\ap)\in \Delta({\ge}{-}1) \ \text{ for all }\ap\in\Pi \}$.
\end{thm}
\begin{proof}
{\sf\bfseries (i)} Suppose that $w=w_{I,\min}$ and $w^{-1}(\ap)\in \Delta(-k)$ for some $\ap\in\Pi$ and 
$k\ge 1$. More precisely, if  $w^{-1}(\ap)=-\gamma$, then $w(\gamma)=-\ap$. Hence $\gamma\in I^k$. 
Assume that $k\ge 2$. Then $\gamma=\gamma'+\gamma''$ with $\gamma'\in I$ and 
$\gamma''\in I^{k-1}$.  Here we would obtain that $-\ap=w(\gamma')+w(\gamma'')$ is a sum of two 
negative roots, which is absurd. Thus, $k\le 1$.

{\sf\bfseries (ii)} Conversely, suppose that $w\in W^0$ has the property that 
$w^{-1}(\ap)\in \Delta({\ge}{-}1)$ for all $\ap\in\Pi$. Set $I=N(w)\cap \Delta(1)$. Then
$I\in \jdod$, because $w\in W^0$. Therefore $\langle I\rangle=N(w_{I,{\rm min}})$ and  
$N(w_I)\subset N(w)$. The last inclusion implies that $w=uw_{I,{\rm min}}$ for some $u\in W$ such that 
$\ell(w)=\ell(u)+\ell(w_{I,{\rm min}})$, see e.g.~\cite[Lemma\,5.1]{som-tym}. 
Assume that $u\ne 1_W$. Then $w=s_\ap u'w_{I,{\rm min}}$ for some 
$\ap\in\Pi$ such that $\ell(u)=1+\ell(u')$ and therefore
\[
     N(w)=N(u'w_{I,{\rm min}})\cup (u'w_{I,{\rm min}})^{-1}(\ap) .
\]
Since $\ell(u'w_{I,{\rm min}})=\ell(u')+\ell(w_{I,{\rm min}})$, we have
$N(u'w_{I,{\rm min}})\supset N(w_{I,{\rm min}})\supset I$. Therefore $(u'w_{I,{\rm min}})^{-1}(\ap)\in \Delta(k)$ 
and here $k\ge 2$.  Then  $w^{-1}(\ap)=-(u'w_{I,{\rm min}})^{-1}(\ap)\in \Delta(-k)$, which contradicts the 
assumption on $w$. Thus, $w=w_{I,{\rm min}}$, and  we are done.
\end{proof}

\begin{thm}    \label{thm:0-max-apriori}
 $W^0_{\rm max}=\{w\in W^0\mid w^{-1}(\ap)\in \Delta({\le}1) \ \text{ for all }\ap\in\Pi \}$.
\end{thm}
\begin{proof}
The proof is similar to the previous one and left to the reader.
\end{proof}
 
Below we point out a relationship between an involution on $\jdod$, involution on $W^0$, 
and the subsets $W^0_{\rm min}$ and $W^0_{\rm max}$.
Let $w_0\in W$ and $\tilde w_0\in W(0)$ be the respective longest elements. It is easily seen that 
if $w\in W^0$, then $w_0w\tilde w_0\in W^0$. Therefore, the mapping 
\[
        w\in W^0  \mapsto i(w):=w_0w\tilde w_0  \in W^0
\]
is a well-defined involution on $W^0$, see \cite{eng}. For any $I\in \jdod$, we have defined the 
{\it dual lower ideal\/} $I^*$ by $I^*=\tilde w_0(\Delta(1)\setminus I)$, see~\cite[Sect.\,2]{ja}.
Note that $\#I+\#I^*=\#\Delta(1)$.

\begin{prop}   \label{lem:min-max-in-W^0}  
For any $w\in W^0$, we have
\par
{\sf\bfseries (i)} \ $(I_w)^*=I_{i(w)}$.

{\sf\bfseries (ii)} \ $ w\in W^0_{\rm min}$ \ if and only if \ $i(w)\in W^0_{\rm max}$. More precisely,
$i(w_{I,{\rm min}})=w_{I^*,{\rm max}}$.     
\end{prop}
\begin{proof}
(i) We have $N(w_0w\tilde w_0)=\Delta^+\setminus N(w\tilde w_0)$. Since 
$\ell(w\tilde w_0)=\ell(w)+\ell(\tilde w_0)$, one also has
$N(w\tilde w_0)=N(\tilde w_0)\cup (\tilde w_0)^{-1} N(w)=\Delta(0)^+\cup \tilde w_0(N(w))$~\cite[Lemma\,5.1]{som-tym}. 
Therefore,
$N(w_0w\tilde w_0)\cap \Delta(1)=\Delta(1)\setminus \tilde w_0(N(w)\cap\Delta(1))$. That is,
$I_{i(w)}=\Delta(1)\setminus \tilde w_0(I_w)=(I_w)^*$.

(ii) Combine part (i), characterisations of $W^0_{\rm min}$ and $W^0_{\rm max}$ in 
Theorems~\ref{thm:0-min-apriori}, \ref{thm:0-max-apriori}, and the following properties of the longest
elements:  $w_0$ takes $\Pi$ to $-\Pi$;  whereas 
$\tilde w_0$ takes each $\Delta(i)$ to itself and also $\Delta(0)^+$ to $-\Delta(0)^+$. 
\end{proof}

For any subset $S\subset W$, define its Poincar\'e polynomial by
$ \displaystyle
  \boldsymbol{S}(t)=\sum_{w\in S}t^{\ell(w)}=\sum_{w\in S}t^{\# N(w)}$. 
The celebrated Kostant-Macdonald identity \cite{macd72} says that 
\beq   \label{eq:KM-ident}
\boldsymbol{W}(t)=\displaystyle\prod_{\gamma\in\Delta^+}
\frac{1-t^{\hot(\gamma)+1}}{1-t^{\hot(\gamma)}} .
\eeq
In particular,  $\# W=\prod_{\gamma\in\Delta^+}\frac{\hot(\gamma)+1}{\hot(\gamma)}$.

\begin{ex}   \label{ex:abelian+Eng}
Let $\Delta=\bigsqcup_{i=-1}^1\Delta(i)$ be an abelian grading. Then
Theorems~\ref{thm:0-min-apriori} and \ref{thm:0-max-apriori}
immediately imply that $W^0_{\rm min}=W^0_{\rm max}=W^0$. Therefore $\#\jdod=\#W^0$. Furthermore,
$i(w)=w$ if and only if $(I_w)^*=I_w$. For any parabolic subgroup $W(0)\subset W$, we have
\[
     \#\{w\in W^0\mid i(w)=w\}=\boldsymbol{W^0}(-1) ,
\]
see~\cite{eng,reiner}. Therefore, in the abelian case, $\boldsymbol{W^0}(-1)$ equals the number of 
self-dual lower ideals in $\Delta(1)$. This has already been proved in \cite{stembr}. In the abelian case, 
$\boldsymbol{W^0}(t)$ coincides with
the rank-generating function for the poset of lower ideals, see e.g.~\cite[Sect.\,3]{ja}, i.e., 
$\boldsymbol{W^0}(t)=\mdt$ and thereby $\eus M_{\Delta(1)}(-1)$ is the number of self-dual lower ideals.
\end{ex}

\begin{rmk}
For the non-abelian $\BZ$-gradings (i.e., if $\Delta(2)\ne\varnothing$), $W^0_{\rm min}$ and  
$W^0_{\rm max}$ are different proper subsets $W^0$. 
Moreover, the polynomials
$\boldsymbol{W^{0}_{\rm min}}(t)$, $\boldsymbol{W^0_{\rm max}}(t)$, and $\mdt$, which have the same value at $t=1$, are different.
For the reader convenience, we compare explicit formulae for all these polynomials:
\begin{gather*}
 \mdt=\sum_{w\in W^0_{\rm min}} t^{\# (N(w)\cap \Delta(1))}=
 \sum_{w\in W^0_{\rm max}} t^{\# (N(w)\cap \Delta(1))} ,
\\
 \boldsymbol{W^{0}_{\rm min}}(t)=\sum_{w\in W^0_{\rm min}} t^{\#N(w)}, \quad
 \boldsymbol{W^0_{\rm max}}=\sum_{w\in W^0_{\rm max}} t^{\#N(w)} . 
\end{gather*}
We have conjectured in \cite[Conjecture\,5.2]{ja} (and verified in many cases) that 
$\eus M_{\Delta(1)}(-1)$ yields the number of self-dual lower ideals in $\Delta(1)$
for {\bf any} $\BZ$-grading. That is, in a sense, $\mdt$ is the most appropriate $t$-analogue of $\#\jdod$. \end{rmk}
\begin{ex}   \label{ex:extra-spe}
Let $\Delta=\bigsqcup_{i=-2}^2\Delta(i)$ be an extra-special grading. As $\Delta(2)=\{\theta\}$, it
follows from Theorem~\ref{thm:0-min-apriori} that, for $w\in W^0$, we have $w\in W^0_{\rm min}$ if and only if 
$w^{-1}(\ap)\ne -\theta$ for all $\ap\in\Pi$, i.e., $-w(\theta)\not\in \Pi$. Likewise, by 
Theorem~\ref{thm:0-max-apriori}, $w\in W^0_{\rm max}$ if and only if $w(\theta)\not\in\Pi$. Hence here 
$W^0=W^0_{\rm min}\cup W^0_{\rm max}$.
As $W(0)$ is the stabiliser of $\theta$ in $W$, we have $W^0{\cdot}\theta=W{\cdot}\theta$,
$\# W^0$ is the number of long roots in $\Delta$, and the non-minimal (or non-maximal) elements
of $W^0$ are parameterised by the set, $\Pi_l$, of long simple roots.
Since the number of long roots is $\#\Pi_l{\cdot}h$ \cite[Chap.\,VI, \S\,1.11, Prop.\,33]{bour}, this yields the equality 
$\#W^0_{\rm min}=\#\Pi_l{\cdot}(h-1)$. The last formula for the number of the ideals/antichains in $\Delta(1)$
was obtained earlier in \cite[Theorem\,4.2]{ja}. 
\end{ex}

\begin{rem}
If a $\BZ$-grading is neither abelian nor extra-special, then $W^0\ne W^0_{\rm min}\cup W^0_{\rm max}$.
\end{rem}
\noindent
Suppose now that the $\BZ$-grading in question is $1$-standard. More precisely, $\Pi=\Pi(0)\cup \Pi(1)$ 
and $\Pi(1)=\{ \tilde\ap \}$. For any $w\in W^0$, we look at the coefficient of $\tilde\ap$ for the roots
$w^{-1}(\ap)$, $\ap\in\Pi$. Namely, write
\[
     w^{-1}(\ap)=k_\ap(w)\tilde\ap+ \sum_{\ap_i\in \Pi(0)}l_i(w)\ap_i 
\]
and consider the mapping $\eta: W^0 \to \BZ^n$, \ $\eta(w)=(k_\ap(w))_{\ap\in\Pi}$.

\begin{thm}    \label{thm:injective-eta}
{\sf\bfseries (i)} \ The mapping $\eta$ is injective; 

{\sf\bfseries (ii)} \ $\eta(W^0_{\rm min})=\{  (k_\ap(w))_{\ap\in\Pi}\mid k_\ap(w)\ge -1\ \text{ for all } \ap\in\Pi\}$;

{\sf\bfseries (iii)} \ $\eta(W^0_{\rm max})=\{  (k_\ap(w))_{\ap\in\Pi}\mid k_\ap(w)\le 1\ \text{ for all } \ap\in\Pi\}$;

\end{thm}
\begin{proof}
(i) \ Let $\{\varpi_{\ap}^\vee\}_{\ap\in\Pi}$ be the fundamental weights of the dual Lie algebra $\g^\vee$ corresponding to $\Pi$.
In other words, $(\ap, \varpi_{\beta}^\vee)=\delta_{\ap\beta}$ for all
$\ap,\beta\in\Pi$, i.e., $\{\varpi_{\ap}^\vee\}_{\ap\in\Pi}$ is the dual basis to $\Pi$.
Then $W(0)$ is the stabiliser of $\varpi_{\tilde\ap}^\vee$ in $W$ and all weights
$w(\varpi_{\tilde\ap}^\vee)$, $w\in W^0$, are different. We have $(w(\varpi_{\tilde\ap}^\vee),\ap)=
(\varpi_{\tilde\ap}^\vee, w^{-1}(\ap))=k_\ap(w)$. Whence $w(\varpi_{\tilde\ap}^\vee)=
\sum_{\ap\in\Pi} k_\ap(w) \varpi_\ap^\vee$.

(ii),\,(iii). This readily follows from Theorems~\ref{thm:0-min-apriori} and \ref{thm:0-max-apriori}, because
$w^{-1}(\ap)\in \Delta(i)$ if and only if $k_\ap(w)=i$.
\end{proof}
\begin{rem}
The above proof suggests to regard $\eta$ as a mapping from $W^0$ to the lattice
$\mathcal L=\{ \sum_{\ap\in\Pi} k_\ap \varpi_\ap^\vee \mid k_\ap\in \BZ\}\simeq \BZ^n$ in $V$. Set also
$\mathcal C_{\ge -1}=\{\sum_{\ap\in\Pi} k_\ap \varpi_\ap^\vee \mid k_\ap \ge -1 \ \forall\ap\in\Pi\}$ and
$\mathcal C_{\le 1}=\{\sum_{\ap\in\Pi} k_\ap \varpi_\ap^\vee \mid k_\ap \le 1 \ \forall\ap\in\Pi\}$.
Then Theorem~\ref{thm:injective-eta} asserts that 
\\[.7ex]
\centerline{ $\eta(W^0_{\rm min})=W{\cdot}\varpi_{\tilde\ap}^\vee\cap \mathcal C_{\ge -1}$ and 
$\eta(W^0_{\rm max})=W{\cdot}\varpi_{\tilde\ap}^\vee\cap \mathcal C_{\le 1}$. }
\\[.5ex]
Thus, the minimal or maximal elements of $W^0$ are in a natural one-to-one correspondence with certain subsets of the $W$-orbit of $\varpi_{\tilde\ap}^\vee$.
\end{rem}

\begin{ex}  \label{ex:abelian-grad} 
The abelian gradings are $1$-standard and then $\varpi_{\tilde\ap}^\vee$ is a minuscule fundamental weight of $\g^\vee$. Then $(\varpi_{\tilde\ap}^\vee, \gamma)\in \{-1,0,1\}$ for all 
$\gamma\in\Delta$~\cite[Ch.\,VIII,\ \S\,7, n$^0$3]{bour7-8}. Consequently,
the whole orbit $W{\cdot}\varpi_{\tilde\ap}^\vee$ belongs to $\mathcal C_{\ge -1}\cap\mathcal C_{\le 1}$.
Here we again obtain that all elements of $W^0$ are both maximal 
and minimal, and therefore 
$\#\anod=\#W^0$. 
\end{ex}

\section{Extreme roots associated with the lower ideals in $\Delta(1)$}
\label{sect:extreme-roots}

Recall that any lower (resp. upper) ideal of a poset $\eus P$ is determined by its maximal (resp. minimal) 
elements. Below, we describe these extreme elements (roots) for the ideals in $\eus P=\Delta(1)$, using 
the corresponding minimal and maximal elements of $W^0$.

\begin{thm}  \label{thm:max-roots} 
If $I\in \jdod$ and  $\gamma\in \Delta(1)$, then 
 $\gamma\in \max(I)$ if and only if $w_{I,{\rm min}}(\gamma)\in -\Pi$.
\end{thm}
\begin{proof} Write $w$ for $w_{I,{\rm min}}$ in this proof. Recall that $\gamma\in I$ if and only if 
$w(\gamma)\in -\Delta^+$.

\noindent
{\sf\bfseries (i)} \ If $\gamma\in I$ and $\gamma\not\in \max(I)$, then $\gamma=\gamma'-\delta$ for some $\gamma'\in I$ 
and $\delta\in \Delta(0)^+$. Then $w(\gamma)=w(\gamma')-w(\delta)$ is a sum of negative roots.

\noindent
{\sf\bfseries (ii)} \ Conversely, if $\gamma\in I$ and $w(\gamma)\not\in -\Pi$, then $w(\gamma)=-\delta_1-\delta_2$, where 
$\delta_i\in\Delta^+$. Hence $-w^{-1}(\delta_1)-w^{-1}(\delta_2)=\gamma\in \Delta(1)$. 
Set $\mu_i=-w^{-1}(\delta_i)$, so that $\gamma=\mu_1+\mu_2$. Without loss 
of generality, we may assume that $\mu_2$ is positive. Let us consider possible 
levels of $\mu_2$ and consequences of that for $\gamma$.

$\mathbf{(1)}$ \ The case in which $\mu_2\in \Delta(0)^+$ is  impossible, since $w(\mu_2)=-\delta_2$ 
and $w\in W^0$.

$\mathbf{(2)}$ \ Suppose that $\mu_2\in \Delta(1)$. Since $w(\mu_2)$ is negative, we have $\mu_2\in I$. 
Furthermore, here $\mu_1\in \Delta(0)$. As in $\mathbf{(1)}$, the case $\mu_1\in \Delta(0)^+$ is
impossible. Hence $-\mu_1\in \Delta(0)^+$ and then $\gamma=\mu_1+\mu_2\prec \mu_2$, i.e.,
$\gamma\not\in\max(I)$.

$\mathbf{(3)}$ \ Suppose that $\mu_2\in \Delta(k)$, $k\ge 2$. Let us show that there is another decomposition
$\gamma=\tilde\mu_1+\tilde\mu_2$ such that  $\tilde\mu_2\in \Delta(\tilde k)$ with $0<\tilde k<k$.
\par
Since $w(\mu_2)$ is negative, we have $\mu_2\in I^k$ by the very definition of $w=w_{I,{\rm min}}$. 
Hence, $\mu_2=\mu'+\mu''$, where $\mu'\in I^{k'}$,
$\mu''\in I^{k''}$, and $k'+k''=k$. As $\gamma=\mu_1+\mu'+\mu''$, we have 
$\mu_1+\mu'\in \Delta$ or $\mu_1+\mu''\in\Delta$, see Lemma~\ref{lem:tri-kornya}. 
By symmetry, it suffices to consider the first possibility.
Then we set $\tilde\mu_1=\mu_1+\mu'$, $\tilde\mu_2=\mu''$,  and $\tilde k=k''$.

Thus, one can gradually descend to the case $\tilde k=1$ and conclude using $\mathbf{(2)}$
that $\gamma\not\in \max(I)$.
\end{proof}

\begin{thm}  \label{thm:min-roots}
For $I\in \jdod$ and  $\gamma\in \Delta(1)$,
we have $\gamma\in \min(I^{\boldsymbol{c}})$ if and only if $w_{I,\text{max}}(\gamma)\in \Pi$.
 \end{thm}
\begin{proof} 
This proof is similar (and ``dual'') to the proof of Theorem~\ref{thm:max-roots}.
Write $w$ for $w_{I,{\rm max}}$ in this proof. Recall that $\gamma\in I^{\boldsymbol{c}}$ if and only if 
$w(\gamma)\in \Delta^+$.
 
\noindent
{\sf\bfseries (i)} \ If $\gamma\in I^{\boldsymbol{c}}\setminus \min(I^{\boldsymbol{c}})$, then $\gamma=\gamma'+\delta$ for some $\gamma'\in I^{\boldsymbol{c}}$ 
and $\delta\in \Delta(0)^+$. Then $w(\gamma)=w(\gamma')+w(\delta)$ is a sum of positive roots.

\noindent
{\sf\bfseries (ii)} \ Conversely, if $\gamma\in I^{\boldsymbol{c}}$ and $w(\gamma)\not\in \Pi$, then 
$w(\gamma)=\delta_1+\delta_2$, where 
$\delta_i\in\Delta^+$. Hence $w^{-1}(\delta_1)+w^{-1}(\delta_2)=\gamma\in \Delta(1)$. 
Set $\mu_i=w^{-1}(\delta_i)$. Without loss 
of generality, we may assume that $\mu_2$ is positive. Let us consider possible levels 
of $\mu_2$ and consequences of that for $\gamma$.

$\mathbf{(1)}$ \ Suppose that $\mu_2\in \Delta(0)^+$. Then $\mu_1\in \Delta(1)$ and 
$w(\mu_1)\in \Delta^+$. Hence $\mu_1\in I^{\boldsymbol{c}}$ and 
$\gamma=\mu_1+\mu_2\not\in\min(I^{\boldsymbol{c}})$.

$\mathbf{(2)}$ \ Suppose that $\mu_2\in \Delta(1)$. Then $\mu_2\in I^{\boldsymbol{c}}$ and $\mu_1\in\Delta(0)$. 

\quad {\bf --} \ If $\mu_1$ is positive, then again $\gamma=\mu_1+\mu_2\not\in\min(I^{\boldsymbol{c}})$.

\quad {\bf --} \ The case in which $\mu_1\in -\Delta(0)^+$ is impossible, since $w(\mu_1)=\delta_1$ and $w\in W^0$.

$\mathbf{(3)}$ \ Suppose that $\mu_2\in \Delta(k)$, $k\ge 2$. Let us show that there is another decomposition
$\gamma=\tilde\mu_1+\tilde\mu_2$ such that  $\tilde\mu_2\in \Delta(\tilde k)$ with $0<\tilde k<k$.
\par
Since $w(\mu_2)\in\Delta^+$, we have $\mu_2\in (I^{\boldsymbol{c}})^k$ by the very definition of $w=w_{I,{\rm max}}$. Hence, $\mu_2=\mu'+\mu''$, where $\mu'\in (I^{\boldsymbol{c}})^{k'}$,
$\mu''\in (I^{\boldsymbol{c}})^{k''}$, and $k'+k''=k$. As $\gamma=\mu_1+\mu'+\mu''$, we have 
$\mu_1+\mu'\in \Delta$ or $\mu_1+\mu''\in\Delta$, see Lemma~\ref{lem:tri-kornya}. 
By symmetry, it suffices to handle the first possibility.
Then we set $\tilde\mu_1=\mu_1+\mu'$, $\tilde\mu_2=\mu''$,  and $\tilde k=k''$.

Thus, one can gradually descend to the case $\tilde k=1$ and conclude using $\mathbf{(2)}$
that $\gamma\not\in \max(I^{\boldsymbol{c}})$.
\end{proof}

\section{Dominant chambers and arrangements of hyperplanes}
\label{sect:dcah}

\noindent
For $\gamma\in\Delta$, let $\eus H_\gamma$ be the hyperplane in $V$ orthogonal to $\gamma$.
Then $\eus A=\{\eus H_\gamma \mid \gamma\in\Delta^+\}$ is the  {\it Coxeter arrangement} 
associated with $\Delta$. The connected components of 
$V\setminus (\bigcup_{\gamma\in\Delta^+}\eus H_\gamma)$ are called (open) {\it chambers}. Each 
chamber 
is an open simplicial cone in $V$, and $W$ acts simply transitively on the set of chambers. 
The {\it dominant\/} open chamber is $\eus C^o=\{v\in V\mid (v,\ap)> 0 \ \ \forall \ap\in\Pi\}$. The 
closure of $\eus C^o$ is denoted by $\eus C$.
If $\eus K',\eus K''$ are two chambers, then the {\it distance\/} between them, $d(\eus K',\eus K'')$, is the 
number of hyperplanes in $\eus A$ that separate them. As is well known,
$d(\eus C, w(\eus C))=\ell(w)$. More precisely, the hyperplane $\eus H_\gamma$ separates $\eus C$ 
and $w(\eus C)$ if and only if $\gamma\in N(w^{-1})$, 
see \cite[Chap.\,VI, \S\,1, Prop.\,17]{bour}.

In this section, we will consider certain sub-arrangements of $\eus A_\Delta$ and their relationship to ideals/antichains in
the poset $\Delta(1)$. The first of them is $\eus A_\Delta(0)=\{\eus H_\gamma \mid \gamma\in\Delta(0)^+\}$, the Coxeter arrangement associated with $\Delta(0)$. The corresponding {\it big\/} dominant chamber is
$\eus C(0)^o=\{v\in V\mid (v,\ap)> 0\ \ \forall\ap\in\Pi(0)\}$ and its closure is denoted by $\eus C(0)$.
It follows readily from the definition of $W^0$
(see Eq.~\eqref{eq:w0-def}), that \ \ 
$w\in W^0$ {\sl if and only if\/} $w^{-1}(\eus C)\subset \eus C(0)$.
In particular, the big dominant chamber $\eus C(0)$ is the union of $\#W^0$ ``small'' chambers.

\begin{thm}   \label{thm:min-max-&-chambers} 
\leavevmode\par  \nopagebreak
{\sf\bfseries (i)} \ The hyperplanes $\eus H_\gamma$, $\gamma\in\Delta(1)$, dissect the cone 
$\eus C(0)$ into  
certain regions (cones) that are in a natural one-to-one correspondence with the ideals of 
$\Delta(1)$ 
({\normalfont and we write $\eus R^o_I$ for the open region corresponding to $I\in \jdod$});

{\sf\bfseries (ii)} \ if\/ $w\in W^0_{\rm min}$, then $w^{-1}(\eus C^o)$ is the unique small chamber in 
$\eus R^o_{I_w}$ that is closest to $\eus C^o$;

{\sf\bfseries (iii)} \ if\/ $w\in W^0_{\rm max}$, then $w^{-1}(\eus C^o)$ is the unique small chamber in 
$\eus R^o_{I_w}$ that is farthest from $\eus C^o$;
\end{thm}
\begin{proof}
(i) \ Given $I\in \jdod$, define the open region (cone), $\eus R_I^o$, corresponding to $I$ as follows:
\[
    \eus R_I^o=\{x\in \eus C(0)^o\mid (x,\gamma)>0 \text{ if } \gamma\not\in I \ \& \ 
    (x,\gamma)<0 \text{ if } \gamma\in I \}.
\]
Using the fact that $\tau: W^0\to \jdod$ is onto, one immediately obtains that 
$\eus R_I^o\ne \varnothing$ for any $I$. Indeed, if $\tau(w)=I$, then
$\eus H_\gamma$ ($\gamma\in\Delta(1)$) separates $\eus C^o$ and $w^{-1}(\eus C^o)$
if and only if $\gamma\in N(w)\cap\Delta(1)=I$. Therefore, $w^{-1}(\eus C^o)\subset \eus R_I^o$.
Furthermore, any chamber $w^{-1}(\eus C^o)$, $w\in W^0$, belongs to  some region $\eus R_I^o$, which 
means that the closed regions $\eus R_I$ ($I\in \jdod$)
exhaust the big dominant chamber $\eus C(0)$.

(ii),(iii) \ This follows from (i) and the fact that $w_{I,\min}$ (resp. $w_{I,\max}$) is the
unique element of minimal (resp. maximal) length in $\tau^{-1}(I)$.
\end{proof}

These properties suggest to consider the sub-arrangement $\eus A_\Delta(0,1)$ of $\eus A_\Delta$ that 
contains only the hyperplanes $\eus H_\gamma$ corresponding to $\gamma\in \Delta(0)^+\cup\Delta(1)$. 
Set $\eta_i=\#\{\gamma\in \Delta(0)^+\cup\Delta(1)\mid\hot(\gamma)=i\}$ and consider the associated
sequence $\mathcal P(0,1)=(\eta_1,\eta_2,\dots )$.

\begin{lm}   \label{lem:partition}
The sequence  $\mathcal P(0,1)$ is a partition, i.e., 
$\eta_1\ge \eta_2\ge \dots$. In addition,
$\eta_1>\eta_2$.
\end{lm}
\begin{proof}
This is a particular case of a more general observation, see \cite[Prop.\,3.1]{som-tym}. However, that
proof consists of a reference to case-by-case and computer computations. For this reason, we provide
a general case-free proof in the Appendix, see Proposition~\ref{prop:appendix}.

Note also that, for the standard gradings, the inequality $\eta_1>\eta_2$ readily stems from the fact that $\Delta(0)^+\cup\Delta(1)$ contains all simple roots, i.e., $\eta_1=\rk\Delta$.
\end{proof}

\begin{conj}   \label{conj:free-arr}
The arrangement $\eus A_\Delta(0,1)$ is free and its exponents are given by the dual partition 
$\mathcal P(0,1)^{\boldsymbol t}$ to $\mathcal P(0,1)$.
\end{conj}

This is a special case of a general conjecture discussed in~\cite{som-tym}. Namely, let $\cI\subset \Delta^+$
be an arbitrary upper ideal and $\eus A_\Delta(\cI^{\boldsymbol c})=\{\mathcal H_\gamma \mid \gamma\not\in\cI\}\subset \eus A_\Delta$. {Sommers} and {Tymoczko} conjecture that the arrangement $\eus A_\Delta(\cI^{\boldsymbol c})$ is free and its exponents are given by the dual partition to $(\lb_1,\lb_2,\dots)$, where
$\lb_i=\#\{\gamma\in \Delta^+\setminus \cI\mid \hot(\gamma)=i\}$. (The string $(\lb_1,\lb_2,\dots)$ is
really a partition, see Proposition~\ref{prop:appendix}.) By
\cite[Theorem\,11.1]{som-tym}, this general conjecture, and thereby Conjecture~\ref{conj:free-arr}, are 
true if $\Delta$ is of type
$\GR{A}{n},\GR{B}{n},\GR{C}{n},\GR{D}{n}$,  and $\GR{G}{2}$.
Using this conjecture, one derives a closed formula for the number of lower ideals (antichains ) in
$\Delta(1)$.

\begin{thm}   \label{thm:number-anod}
It follows from Conjecture~\ref{conj:free-arr}  that
\beq   \label{eq:chislo-ide}
    \#\bigl(\jdod\bigr)= \#\anod=\prod_{\gamma\in\Delta(1)}\frac{\hot(\gamma)+1}{\hot(\gamma)} .
\eeq
\end{thm}
\begin{proof}
Let  $b_1,\dots,b_n$ be the exponents of the free arrangement $\eus A_\Delta(0,1)$. By the factorisation 
result of Terao (see \cite[Theorem\,4.137]{OS},
the characteristic polynomial of $\eus A_\Delta(0,1)$ is $\chi_{(0,1)}(t)=\prod_{i=1}^n (t-b_i)$.
By a theorem of Zaslavsky \cite{zasl}, the total number of regions of  $\eus A_\Delta(0,1)$ equals
$(-1)^n\chi_{(0,1)}(-1)=\prod_{i=1}^n (b_i+1)$.
By definition of the dual partition, if $\eta_i=\#\{\gamma\in \Delta(0)^+\cup \Delta(1)\mid \hot(\gamma)=i\}$, then $\eta_i-\eta_{i+1}$ is the number of exponents that are equal to $i$.
Therefore,
\[
   \prod_{i=1}^n (b_i+1)=\prod_{\gamma\in \Delta(0)^+\cup \Delta(1)}\frac{\hot(\gamma)+1}{\hot(\gamma)}. 
\]
Since the arrangement $\eus A_\Delta(0,1)$ is $W(0)$-invariant and $\eus C(0)$ is a fundamental domain for
the $W(0)$-action, the number of regions inside $\eus C(0)$
equals $\prod_{i=1}^n (b_i+1)/\#W(0)$. On the other hand, the  
Kostant-Macdonald identity~\eqref{eq:KM-ident} implies that 
$\#W(0)=\prod_{\gamma\in \Delta(0)^+}\frac{\hot(\gamma)+1}{\hot(\gamma)}$. Combining all these 
formulae, we conclude that the number of  regions of $\eus A_\Delta(0,1)$ inside $\eus C(0)$
equals $\prod_{\gamma\in \Delta(1)}\frac{\hot(\gamma)+1}{\hot(\gamma)}$. Finally, by Theorem~\ref{thm:min-max-&-chambers}(i), the last number also gives the number of antichains (ideals) in $\Delta(1)$. 
\end{proof}

\begin{rmk}
Formula~\eqref{eq:chislo-ide} for $\#\anod$ appears already in \cite{ja} as a consequence of a 
general conjectural formula for $\mdt$~\cite[Conj.\,5.1]{ja}. Now, our theory of minimal/maximal elements 
in $W^0$, a relationship to arrangements, and partial results of \cite{som-tym} allow us to conclude 
that~\eqref{eq:chislo-ide} holds for all classical cases and $\GR{G}{2}$. However, the present approach does not provide new information on $\mdt$, because there seems to be no relationship between  the arrangement $\eus A_\Delta(0,1)$ and the rank-generating function $\mdt$.
\end{rmk}
\begin{ex}    \label{ex:(0,1)-abelian}
In the abelian case, we have $\eus A_\Delta(0,1)=\eus A_\Delta$ and the exponents of the Coxeter 
arrangement $\eus A_\Delta$ are the usual exponents of the Weyl group $W$~\cite[Theorem~6.60]{OS}. 
Hence $\chi_{\eus A_\Delta}(t)=\prod_{i=1}^n (t-m_i)$ and
$(-1)^n\chi_{\eus A_\Delta}(-1)=\prod_{i=1}^n (m_i+1)=\# W$,  as required.
\end{ex}

\noindent
As usual, we arrange the exponents in the non-decreasing order:
$1=m_1\le m_2\le\dots\le m_n=h-1$. If $n\ge 2$, then $m_1<m_2$ and $m_{n-1}< m_n$.

\begin{ex}    \label{ex:(0,1)-extra-spec}
In the extra-special case, $W(0)$ is the stabiliser of $\theta$ and 
$\eus A_\Delta(0,1)$ is just the {\it deleted arrangement\/} 
$\eus A'=\eus A_\Delta\setminus \eus H_\theta$. It is known that $\eus A'$ is free and the exponents of 
$\eus A'$ are $m_1,\dots,m_{n-1},m_n-1$ (combine Theorems~4.51 and 6.104 in~\cite{OS}).
Therefore $(-1)^n \chi_{\eus A'}(-1)=(m_1+1)\dots (m_{n-1}+1)m_n=\# W{\cdot}\frac{h-1}{h}$. Since
$\# W/\# W(0)$ is the number of long roots in $\Delta$, the number of the $W(0)$-dominant 
regions of $\eus A'$ is 
\[
   \frac{\# W}{\# W(0)}\cdot \frac{h-1}{h}=\#\Pi_l{\cdot}h\cdot \frac{h-1}{h}=\#\Pi_l{\cdot}(h-1),
\]
which is the number of antichains in $\Delta(1)$. This was computed earlier in \cite[Section\,4]{ja}, see
also Example~\ref{ex:extra-spe}.
\end{ex}

\begin{ex}   \label{ex:E7-ap7}
For the $1$-standard $\BZ$-grading of $\g=\GR{E}{7}$ with $\Pi(1)=\{\ap_7\}$, we have
$\g(0)\simeq \mathfrak{gl}(7)$ and  $\g(1)=\wedge^3(\BC^7)$ is the third fundamental representation.
Here the numbering of $\Pi$ follows \cite[Tables]{t41}. Then
\[
   \mathcal P(0,1)=(7,6^4,5^2,4^2,3,2,1,1) \ \text{ and } \ \mathcal P(0,1)^{\boldsymbol t}=(13,11,10,9,7,5,1) . 
\]
Therefore, the conjectural exponents of $\eus A_\Delta(0,1)$ are
$1,5,7,9,10,11,13$ and then the number of lower ideals in $\Delta(1)$ is $252$.
\end{ex}

\section{Affine versus finite theory}
\label{sect:versus}

\noindent
In this section, we compare the theory of upper (or ad-nilpotent) 
ideals of $\Delta^+$ (the {\sl affine theory}) and our theory of lower ideals in $\Delta(1)$ related to a 
$\BZ$-grading of $\Delta$ (the {\sl finite theory}).

We begin with the necessary notation. 
Recall that $V=\oplus_{i=1}^n{\BR}\ap_i$ and 
$(\ ,\ )$ is a $W$-invariant inner product on $V$. As usual,
$\mu^\vee=2\mu/(\mu,\mu)$ is the coroot
for $\mu\in \Delta$ and $\mathcal Q^\vee=\oplus _{i=1}^n {\BZ}\ap_i^\vee$  
is the {\it coroot lattice\/} in $V$.
Letting $\widehat V=V\oplus \BR\delta\oplus \BR\lb$, we extend
the inner product $(\ ,\ )$ on $\widehat V$ so that $(\delta,V)=(\lb,V)=(\delta,\delta)=(\lb,\lb)=0$ and 
$(\delta,\lb)=1$. Set $\ap_0=\delta-\theta$.

Then 
\begin{itemize}
\item[] \ 
$\HD=\{\Delta+k\delta \mid k\in \BZ\}$ is the set of affine
(real) roots; 
\item[] \ $\widehat{\Delta}^+= \Delta^+ \cup \{ \Delta +k\delta \mid k\ge 1\}$ is
the set of positive affine roots; 
\item[] \ $\widehat \Pi=\Pi\cup\{\ap_0\}$ is the corresponding set
of affine simple roots. 
\end{itemize}
For any $\gamma\in \HD$, the reflection $s_\gamma\in GL(\HV)$ is defined in the usual way, via the 
extended inner product, and
the affine Weyl group, $\HW$, is the subgroup of $GL(\HV)$
generated by the reflections $s_\ap$, $\ap\in\HP$. As is well known, $\HW$ is also a semi-direct product of
$W$ and $\mathcal Q^\vee$ \cite{bour,hump}. It follows that  $\HW$ has two natural actions:

(a) \ the linear action on $\HV$;

(b) \ the affine-linear action on $V$.
\\
Using the linear action, one defines the inversion
set $\widehat N(w)=\{\gamma\in \HD^+\mid w(\gamma) \in -\HD^+ \}$ and the length $\hat\ell(w)=\#\widehat N(w)$ 
for any $w\in \HW$.

The affine theory is well-developed, and we present below notable correlations with 
results of this article. An overview of the ``affine'' results discussed below 
can also be found in \cite[Section\,2]{losh}. 

{\it\bfseries 1)} \ By the very definition, $\HD$ is $\BZ$-graded, with $\HD(k)=\Delta+k\delta$, $k\in \BZ$.
Extending our previous terminology to the affine case, one can say that this $\BZ$-grading is $1$-standard. The unique affine simple root in $\HD(1)$ is $\ap_0$ and the parabolic subgroup
$\HW(0)$ is just $W$. Accordingly, the set of minimal length coset representatives is
\[
     \HW^0=\{w\in \HW \mid w(\ap) \in\HD^+ \ \text{ for all } \ \ap\in \Pi\}    
\]
(such elements of $\HW$ are called {\it dominant} in \cite{losh}.)
Let $\cI$ be an {\sl upper\/} ideal of the poset $(\Delta^+, \curle)$, i.e., $\cI\in \eus J_+(\Delta^+)$.
The affine theory 
gets off the ground when one replaces $\cI$ with
$\delta-\cI=\{\delta-\gamma \mid \gamma\in\cI\}\subset \HD(1)$ and seeks for a characterisation of 
$\delta-\cI$ is terms of $\HW$, or rather, in terms of $\HW^0$. 
Note that $\delta-\cI$ becomes a {\sl lower\/} ideal in the negative part of $\HD(1)\simeq \Delta$. 

{\it\bfseries 2)} \ Given  $\cI\in \eus J_+(\Delta^+)$, the first basic result is that there is a unique element 
$w_{\cI,{\rm min}}\in \HW^0$ of 
{\sf\bfseries minimal\/} length such that  $\HN(w_{\cI,{\rm min}})\cap \HD(1)=\delta-\cI$. Namely, 
\beq    \label{eq:CP-inversion}
\HN(w_{\cI,{\rm min}})=\bigcup_{k\ge 1}(k\delta-\cI^k)=\bigcup_{k\ge 1}(\delta-\cI)^k  .
\eeq
The key point is to prove that the RHS is a bi-convex subset of $\HD^+$, see \cite[Sect.\,2]{cp1}. Hence our Theorem~\ref{thm:biconvex-} is a ``finite'' analogue of that result.
Then the set of minimal elements of $\cI$ (called {\it generators} of $\cI$ in \cite{duality,losh}), i.e., maximal elements 
of $\delta-\cI$ can be characterised via $w_{\cI,{\rm min}}$, see~\cite[Theorem\,2.2]{duality}.
The corresponding ``finite'' assertion is our Theorem~\ref{thm:max-roots}. 

{\it\bfseries 3)} \ Since $\HW$ and $\HD$ are infinite, one cannot always provide an element 
$w_{\cI,{\rm max}}\in \HW^0$ of {\sf\bfseries maximal\/} length such that 
$\HN(w_{\cI,{\rm max}})\cap \HD(1)=\delta-\cI$. {Sommers} proves \cite{som05} that such a maximal element 
exists if and only if  $\cI\subset \Delta^+\setminus \Pi$. 
In that case, $w_{\cI,{\rm max}}$ can be used for describing the maximal elements of $\Delta^+\setminus\cI$,
i.e., the minimal elements of $\HD(1)\setminus (\delta-\cI)$, see \cite[Cor.\,6.3]{som05}. Our 
Theorems~\ref{thm:biconvex-} and \ref{thm:min-roots} provide finite analogues
of this for {\sl all\/} lower ideals in $\Delta(1)$.

{\it\bfseries 4)} \ In the finite case, $\Delta(1)$ is the weight poset of a {\sl weight multiplicity free\/} 
representation of $\g(0)$, and the maximal and minimal elements in $W^0$ exist for all lower ideals.
But the adjoint representation of $\g$ is {\sl not\/}  weight multiplicity free (unless $\g=\tri$). Therefore, 
in the affine case, one considers only the weight multiplicity free part of $\g$ corresponding to $\Delta^+$.
A related disadvantage is that $\Delta^+\setminus \cI$ shouldn't be called a ``lower ideal'' and that
$w_{\cI,{\rm max}}$ does not always exists, see {\it\bfseries 3)} above.

{\it\bfseries 5)} \ Among the advantages of the affine case are the following:
\begin{itemize}
\item $\HW=W\ltimes \mathcal Q^\vee$ is a semi-direct product having two related actions (on $V$ and $\HV$);
\item $\delta$ is a $\HW$-invariant element of $\HV$ and all the pieces $\HD(k)$ are isomorphic;
\end{itemize}
These properties often help in computations and allow to achieve more complete results. 
On the other hand, an advantage of the finite theory is that both $W$ and $W(0)$ contain the elements 
of maximal length, which yields a natural involution on $W^0$ and provides a relationship between
$W^0_{\rm min}$ and $W^0_{\rm max}$ in Proposition~\ref{lem:min-max-in-W^0}.

{\it\bfseries 6)} \ There are at least two approaches to computing the total number of upper ideals (antichains) 
in $\Delta^+$, which are discussed below.

(6a) \ There is a natural bijection between $\eus J_+(\Delta^+)$ and the $W$-dominant regions of the 
{\it Catalan arrangement\/} 
\[
   \mathsf{Cat}(\Delta)=\{\eus H_{\gamma,k}\mid \gamma\in\Delta^+, \ \ k=-1,0,1\},
\]
where $\eus H_{\gamma,k}=\{v\in V\mid (\gamma,v)=k\}$. Then an explicit formula for the characteristic 
polynomial of $\mathsf{Cat}(\Delta)$ yields a formula for $\# \eus J_+(\Delta^+)$, see~\cite{ath1}. A finite
counterpart of this approach is implemented in Section~\ref{sect:dcah}, $\mathsf{Cat}(\Delta)$ being 
replaced with $\eus A_\Delta(0,1)$. In particular, the ``finite'' analogue of the above bijection is our 
Theorem~\ref{thm:min-max-&-chambers}.

(6b) \  There is a natural bijection between the set of minimal elements in $\HW^0$, denoted 
$\HW^0_{\rm min}$, and the points of certain convex polytope $\mathcal{D}_{\rm min}\subset V$ lying in 
$\mathcal Q^\vee$~\cite[Prop.\,3]{cp2}. This polytope is $\HW$-conjugate to a dilated fundamental alcove of $\HW$, and the number 
\[
    \#(\mathcal{D}_{\rm min}\cap \mathcal Q^\vee)= \# \HW^0_{\rm min}=\#\eus J_+(\Delta^+)
\]
can be computed via a result of Haiman, see \cite[Section\,3]{cp2} for details. To construct a bijection
$\HW^0_{\rm min} \stackrel{1:1}{\longleftrightarrow} \mathcal{D}_{\rm min}\cap \mathcal Q^\vee$,
{Cellini} and {Papi} use the semi-direct product structure of $\HW$. However, one can notice that the following
synthetic procedure works. If $w_{\cI,{\rm min}}$ is defined by \eqref{eq:CP-inversion} and 
`$\ast$' denotes the affine-linear action of $\HW$, then the point of $\mathcal Q^\vee$ corresponding to 
$\cI$ is merely $w_{\cI,{\rm min}}\ast 0$. 

{\bf Warning}. {Cellini} and {Papi}~\cite{cp1,cp2} give the definition of the inversion set $\widehat N(w)$
with the inverse of $w\in\HW$.
Therefore, their minimal element corresponding to $\cI$ is the inverse of ours, and hence the
points of $\mathcal Q^\vee$ corresponding to  $\HW^0_{\rm min}$ are also different. 

Since $W=\HW(0)$ is the stabiliser of $0\in V$ w.r.t. the affine-linear action, a finite analogue of the 
Cellini-Papi bijection is the following. 
Suppose that  a $\BZ$-grading of $\Delta$ is $1$-standard and $\Pi(1)=\{\tilde\ap\}$. Then
$W(0)$ is the stabiliser of the fundamental weight $\varpi^\vee_{\tilde\ap}$ 
(Section~\ref{sect:min-and-max-els}) and we need the cardinality of 
$W^0_{\rm min}{\cdot}\varpi^\vee_{\tilde\ap}$. This subset of the orbit $W{\cdot}\varpi^\vee_{\tilde\ap}=
W^0{\cdot}\varpi^\vee_{\tilde\ap}$
is explicitly described,  see Theorem~\ref{thm:injective-eta} and Remark afterwards, but we are unable (yet)  to infer from this a way to compute the 
cardinality.

\begin{rmk}
There are many other aspects of the affine theory that are not mentioned above. Developing their 
``finite'' counterparts can (and will) be the subject of forthcoming publications.
\end{rmk}

\appendix
\section{A partition associated with an upper ideal of $\Delta^+$} 
\label{app:A}

Let $\mathcal I$ be an upper ideal of the poset $(\Delta^+, \curle)$ and $\mathcal I^c=\Delta^+\setminus \mathcal I$. Define 
\[
    \lb_i=\# \{\gamma\in \mathcal I^c\mid \hot(\gamma)=i\} .
\]
Our goal is to give a case-free proof of the following observation, see \cite[Prop.\,3.1]{som-tym}.
\begin{prop}   \label{prop:appendix}
The sequence $(\lb_1,\lb_2,\dots)$ is a partition of the number of roots of $\mathcal I^c$. That is, $\lb_1\ge\lb_2\ge\dots$. \ 
Moreover, if $\mathcal I\ne\Delta^+$, then $\lb_1>\lb_2$.
\end{prop}
\begin{proof}
We use some properties of a principal nilpotent element in the corresponding simple Lie algebra $\g$.
Recall that $\g=\ut \oplus\te\oplus\ut^-$ is a fixed triangular decomposition,  
$\Delta^+$ is the set of $\te$-roots in $\ut $, and $\g_\gamma$ is the roots space corresponding to 
$\gamma\in\Delta$. Take 
$e=\sum_{\ap\in\Pi}e_\ap$, where $e_\ap$ is a nonzero element of $\g_\ap$. After work of {Dynkin} and
{Kostant} in 1950's, it is known that $e$ is a principal nilpotent element of $\g$. Specifically, we need the following properties of the centraliser $\z_\g(e)$ of $e$:

\centerline{ $\z_\g(e)\subset \ut $ \ and \ $\dim \z_\g(e)=n=\rk \g$.}

\noindent
The $\te$-roots in the derived subalgebra $\ut'=[\ut ,\ut ]$ are exactly the non-simple positive roots,
 hence  $\ut'$ is of codimension $n$ in $\ut $. Combining the above properties, we see that
the mapping
$ \ad (e):  \ut  \to \ut'$
is onto. Moreover, both vector spaces are graded:
\[
  \ut =\bigoplus_{i=1}^{h-1} \ut\langle i\rangle \ \text{ and } \ 
  \ut'=\bigoplus_{i=2}^{h-1} \ut\langle i\rangle,
\] 
where $\ut\langle i\rangle=\bigoplus_{\gamma: \hot(\gamma)=i}\g_\gamma$, and $\ad(e)$ is a
homomorphism of degree $1$.
Let $\ce_{\mathcal I}=\bigoplus_{\gamma\in\mathcal I}\g_\gamma$ be the $\be$-stable subspace of 
$\ut $ corresponding to $\mathcal I$.
The quotient spaces  $\ut /\ce_{\mathcal I}$ and  $\ut'/(\ce_{\mathcal I}\cap \ut')$ inherit the above 
grading and  the commutative diagram 
\[
\begin{array}{ccc}  \ut  & \stackrel{\ad(e)}{\longrightarrow} & \ut' \\
\downarrow && \downarrow \\
\ut /\ce_{\mathcal I} & \stackrel{\ad(e)}{\longrightarrow} & \ut'/(\ce_{\mathcal I}\cap \ut')
\end{array}
\]
shows that the map in the bottom row is also graded surjective, of degree $1$. Furthermore, let 
$\tilde\ut\langle i\rangle$ be the component of grade $i$ in $\ut /\ce_{\mathcal I}$. Then
$\ut /\ce_{\mathcal I}=\bigoplus_{i\ge 1}\tilde\ut\langle i\rangle$,
$\ut'/(\ce_{\mathcal I}\cap \ut')=\bigoplus_{i\ge 2}\tilde\ut\langle i\rangle$, 
and $\dim \tilde\ut\langle i\rangle=\lb_i$. Consequently, the graded surjectivity implies that 
$\lb_i\ge \lb_{i+1}$ for all $i$.

Finally, if $\ce_{\mathcal I}\ne \ut$, then $\tilde\ut\langle 1\rangle\ne 0$, and the image of 
$e\in\ut\langle 1\rangle\subset \ut$ in $\tilde\ut\langle 1\rangle \subset \ut /\ce_{\mathcal I}$ is a {\it nonzero\/} 
element in the kernel of 
$\ad(e)$. Hence $\lb_1> \lb_2$.
\end{proof}

\end{document}